\documentclass[11pt,a4paper,twoside]{amsart}
\usepackage{amssymb,amsmath,amsthm,graphicx,color
}
\theoremstyle{plain}
\newtheorem{theorem}{Theorem}[section]
\newtheorem{definition}[theorem]{Definition}
\newtheorem{lemma}[theorem]{Lemma}

\newtheorem{proposition}[theorem]{Proposition}
\newtheorem{assumption}[theorem]{Assumption}
\theoremstyle{remark}
\newtheorem{remark}[theorem]{Remark}

\allowdisplaybreaks[4]

\usepackage{amsrefs}

\numberwithin{equation}{section}

\newcommand{\C}{\mathbb{C}}
\newcommand{\R}{\mathbb{R}}
\newcommand{\Z}{\mathbb{Z}}

\newcommand{\F}{\mathcal{F}}

\renewcommand{\Im}{\operatorname{Im}}
\renewcommand{\Re}{\operatorname{Re}}
\newcommand{\I}{\infty}
\newcommand{\abs}[1]{\left\lvert #1\right\rvert}
\newcommand{\wha}[1]{\widehat{#1}}

\newcommand{\norm}[1]{\left\lVert #1\right\rVert}
\newcommand{\Lebn}[2]{\left\lVert #1 \right\rVert_{L^{#2}}}

\newcommand{\Jbr}[1]{\left\langle #1 \right\rangle}

\def\Sch{{\mathcal S}} 
\def\({\left(}
\def\){\right)}
\def\<{\left\langle}
\def\>{\right\rangle}
\def\le{\leqslant}
\def\ge{\geqslant}

\def\d{{\partial}}

\def \v{\varphi}
\def \f{\phi}
\def \e{\varepsilon}
\def \l{\lambda}

\def \D{\Delta}

\def \pa{\partial}
\def \n{\nabla}
\def \s{\sigma}
\def \a{\alpha}

\def \n{\nabla}
\def \t{\theta}

\def \F{\mathcal{F}}

\newcommand{\eps}{\varepsilon}

\DeclareMathOperator{\supp}{supp}

\newcommand{\todayd}{\the\year/\the\month/\the\day}

\theoremstyle{definition}

\begin{document}
\title[Critical homogeneous NLS]
{Nonexistence of scattering and modified scattering states for some nonlinear Schr\"odinger equation with critical homogeneous nonlinearity}

\author[S. Masaki]{Satoshi MASAKI}
\address[]{Division of Mathematical Science, Department of Systems Innovation, Graduate School of Engineering Science, Osaka University, Toyonaka, Osaka, 560-8531, Japan}
\email{masaki@sigmath.es.osaka-u.ac.jp}
\author[H. Miyazaki]{Hayato MIYAZAKI}
\address[]{Advanced Science Course, Department of Integrated Science and Technology, National Institute of Technology, Tsuyama College, Tsuyama, Okayama, 708-8509, Japan}
\email{miyazaki@tsuyama.kosen-ac.jp}

\date{}

\maketitle
\vskip-5mm
\begin{abstract}
We consider large time behavior of solutions to the nonlinear Schr\"odinger equation with a homogeneous nonlinearity of the critical order which is not necessarily a polynomial. 
We treat the case in which the nonlinearity contains non-oscillating factor $|u|^{1+2/d}$.
The case is excluded in our previous studies. 
It turns out that there are no solutions that behave like a free solution with or without logarithmic phase corrections.
We also prove nonexistence of an asymptotic free solution in the case that the gauge invariant nonlinearity is dominant, and give a finite time blow-up result. 
\end{abstract} 
\section{Introduction}

In this paper, we consider large time behavior of solutions to
 nonlinear Schr\"odinger equation
\begin{equation}\label{eq:NLS}\tag{NLS}
	i \pa_t u + \Delta u = F(u), 
\end{equation}
where $(t,x) \in \R^{1+d}$ and $u=u(t,x)$ is a complex-valued unknown function. 
The nonlinearity $F$ is homogeneous of degree $1+2/d$, that is, $F$ satisfies the condition
\begin{equation}\label{eq:cond1}
	F(\lambda u) = \lambda^{1+\frac2d} F(u)
\end{equation}
for any $u\in \C$ and $\lambda >0$. 

It is known that the degree $1+2/d$ in the assumption \eqref{eq:cond1} is critical in view of large time behavior.
More precisely, the behavior of a solution depends on the shape of the nonlinearity \cite{Oz,GO,MTT,STo,HNST,HNW}.
In \cite{MM2,MMU}, we introduce a decomposition of the nonlinearity 
\begin{equation}\label{eq:decomp}
	F(u) = g_0 |u|^{1+\frac2d} + g_1 |u|^{\frac2d}u + \sum_{n\neq 0,1} g_n |u|^{1+\frac2d-n}u^n
\end{equation}
with the coefficients
\begin{equation}\label{eq:gn}
	g_n = \frac1{2\pi} \int_0^{2\pi} F(e^{i\theta}) e^{-in\theta} d\theta
\end{equation}
and show if $g_0=0$ and $g_1 \in \R$ then the equation \eqref{eq:NLS} admits a solution 
which asymptotically behaves like
\begin{equation}\label{eq:masymptotic}
	u_{\mathrm{ap}} (t) = (2it)^{-\frac{d}{2}}  e^{i \frac{|x|^2}{4t}} \wha{u_{+}}\(\frac{x}{2t}\) \exp \( -i g_1 \left| \wha{u_{+}}\(\frac{x}{2t}\) \right|^{\frac{2}{d}} \log t \)
\end{equation}
as $t\to\I$ for suitable function ${u_+}$, under some summability assumption on $\{g_n\}_n$.
In particular, if $g_0=g_1=0$ then there exists an asymptotically free solution.

In this paper, we consider the case $g_0\neq 0$.
Remark that we may let $g_0=1$ without loss of generality by change of variable.
The behavior of the solutions is studied in some specific cases such as $d=2$ and $F(u)=2(\Re u)^2$ in \cite{HN06FE}.
However,
it seems difficult to predict typical behavior in a general setting
because even small data global existence is not always true \cite{IW} (see also \cite{II2, FO}).
Further, another critical notion of the power of the nonlinearity is reported in \cite{HN14}.
According to these facts,
we do not try to give a behavior in terms of $\{g_n\}_n$ in this paper,
but instead deny the existence of a solution that behaves like a free solution or a free solution with a logarithmic phase correction, that is, behaves like \eqref{eq:masymptotic}.
This is a complementary study of \cite{MM2,MMU}, and is an extension of \cite{Sh,STs}.

\subsection{Nonexistence of a modified scattering state}
To state the results, we introduce notations.
Set $\Jbr{a}=(1+|a|^2)^{1/2}$ for $a \in \C$ or $a\in \R^d$. For $s, m \in \R$,
the weighted Sobolev space on $\R^{d}$ is defined by $H^{m,s} = \{u \in \Sch'(\R^d)\;  ;\; \Jbr{i\n}^m \Jbr{x}^s u \in L^2(\R^d) \}$. 

We first give the definition of a solution.
\begin{definition}[Solution]\label{def:sol}
Let $I \subset \R$ be an interval. 
We say a function $u(t,x): I\times \R^d\to \C$ is a solution to \eqref{eq:NLS} on $I$ if $u(t)$ belongs to
\[
	C_t(I;L^2_x(\R^d)) \cap L^{\frac{2(d+2)}{d}}_{t,\mathrm{loc}}(I; L_x^{\frac{2(d+2)}{d}}(\R^d))
\]
and satisfies  
\[
	u(t_2) = U(t_2-t_1) u(t_1) - i \int_{t_1}^{t_2} U(t_2 -s) F(u(s)) ds
\]
in $L^2(\R^d)$ for any $t_1,t_2 \in I$, where $U(t)=e^{it\Delta}$ is the free Schr\"odinger group.
\end{definition}

For $t\in \R \setminus\{0\}$, we let
unitary operators $M(t)$ and $D(t)$ on $L^2(\R^d)$ by
\[
	[M(t)f](x) = e^{i\frac{|x|^2}{4t}} f(x),\quad
	[D(t)f](x) = (2t)^{-\frac{d}2} f\(\frac{x}{2t}\).
\]
For a number $\lambda \in \R$ and a function $u_+ \in L^2(\R^d)$, we let
\begin{equation}\label{eq:Vl}
	V_\lambda(t)=V_\lambda(t,x;u_{+}) = e^{-i \frac{d\pi}4}(M(t) D(t) [\wha{u_+}\exp(-i\l |\wha{u_+}|^{2/d} \log t)])(x),
\end{equation}
where $\wha{u_+}$ denotes the Fourier transform $(2\pi)^{-\frac{d}2} \int_{\R^d} e^{-ix\cdot\xi} u_+(x) dx$.
Remark that $V_{g_1}(t)$ is the same asymptotic profile as in \eqref{eq:masymptotic}.

Our main result is the following.
\begin{theorem}[No scattering nor modified scattering] \label{thm:main}
Let $d\ge1$.
Suppose that $\{g_n\}_n \in \ell^1(\Z)$ and $g_0 =1$. 
If a solution $u(t)$ to \eqref{eq:NLS} on $[T,\I)$, $T\in \R$, satisfies
\begin{align}
\lim_{t \to \I} \Lebn{u(t)-V_\lambda(t)}{2} = 0 \label{asmp:1}, \\
\lim_{t \to \I} t^{\frac{d}{2(d+2)}} \norm{u(\cdot) - V_\lambda(\cdot)}_{L^{\frac{2(d+2)}{d}}_{t,x}([t, \I)\times \R^d)} = 0, \label{asmp:2}
\end{align}
for some $u_+ \in H^{0,\frac{d}{d+2}}(\R^d)$ and some $\lambda \in \R$, where $V_\lambda(t)$ is given in \eqref{eq:Vl},
then $u_{+} \equiv 0$.
\end{theorem}

\begin{remark}
As mentioned above, if $g_0\in \C \setminus \{0\}$ then we may let $g_0=1$ by change of variable.
\end{remark}

\begin{remark}
If $g(\theta)=F(e^{i\theta})$ is Lipschitz continuous, we can construct a unique local solution for any given $L^2$ data by a standard contraction argument.
Remark that our assumption $\{g_n\}_n \in \ell^1(\Z)$ is weaker than the Lipschitz continuity of $g(\theta)$. Indeed, $g(\theta)=|\cos \theta|^{1/2}$ is such an example.
The case corresponds to $F(u)=|\Re u|^\frac12 |u|^{\frac12+\frac2d}$, and $g_n=O(|n|^{-3/2})$ (see \cite{MMU,MS5}).
\end{remark}

\begin{remark}
When $\l=0$, the assumptions \eqref{asmp:1} and \eqref{asmp:2} are equivalent to
\begin{align}
	\lim_{t \to \I} \Lebn{u(t)-U(t)u_+}{2} &= 0, \label{asmp:3}\\
	\lim_{t \to \I} t^{\frac{d}{2(d+2)}} \norm{u(\cdot) - U(\cdot)u_+}_{L^{\frac{2(d+2)}{d}}_{t,x}([t, \I)\times \R^d)} &= 0,
	\label{asmp:4}
\end{align}
respectively, as long as  $u_+ \in H^{0,d/(d+2)}$ (see Lemma \ref{lem:equivalence}).
Hence, our theorem is a generalization of \cite{Sh,STs}.
\end{remark}
\begin{remark}
Not only the asymptotic profile of the form \eqref{eq:Vl} but also 
profiles with more general phase correction term can be treated (see 
Theorem \ref{thm:general}).
\end{remark}

Our argument is also applicable to the case $g_0=0$ and $g_1 \neq 0$.
We are able to prove the nonexistence of an asymptotic free solution.
\begin{theorem}[No scattering] \label{thm:Barab}
Let $d\ge1$.
Suppose that $\{g_n\}_n \in \ell^1(\Z)$, $g_0 =0$, and $g_1\neq 0$. 
If a solution $u(t)$ to \eqref{eq:NLS} on $[T,\I)$, $T\in \R$, satisfies
\eqref{asmp:3} and \eqref{asmp:4}
for some $u_+ \in H^{0,\frac{d}{d+2}}(\R^d)$,
then $u_{+} \equiv 0$.
\end{theorem}

\begin{remark}
This result can be compared with that by Strauss \cite{St} (see also Barab \cite{B}).
In \cite{St,B}, the case $g_n=\pm \delta_{n1}$ is treated.
The assumptions \eqref{asmp:2} and $u_+ \in H^{0,\frac{d}{d+2}}(\R^d)$ are not used.
However, their argument requires the assumption $\norm{U(t) u_+}_{L^\I} = O(t^{-d/2})$ as $t\to\I$.
It is not clear which assumption is stronger.
\end{remark}

\subsection{Finite time blowup}
As mentioned above, when $g_0=1$ we may not expect even global existence for small data.
By the test function method introduced by \cite{QZ1, QZ2}, we obtain the following blowup result as long as $|u|^{1+2/d}$ is dominant. 

To state the result, we introduce notion of a weak solution. 
\begin{definition}[weak solution]\label{def:wsol}
Suppose that $F(z)$ is locally uniformly bounded. 
We say a function $u(t,x) \in \mathcal{S}'((-\I,T)\times \R^d)$
is a weak solution to \eqref{eq:NLS} with initial condition $u(0,x) = u_0(x) \in L^1_{\mathrm{loc}}(\R^d)$ on $[0,T)$,
$T>0$, if $u \in L^{{(d+2)}/{d}}_{\mathrm{loc}} ( (0, T) \times \R^d)$
 and the identity
\begin{multline*}
	\int_{(0,T) \times \R^d} u(t,x)(-i\d_t \psi(t,x) + \Delta\psi(t,x)) dxdt\\
	= i\int_{\R^d} u_0(x) \psi(0,x)  dx + \int_{(0,T) \times \R^d} F(u(t,x))\psi(t,x) dxdt
\end{multline*}
holds for any test function $\psi \in C_0^\I ((-\I,T) \times \R^d)$.
\end{definition}
Note that a solution (in the sense of Definition \ref{def:sol})
on $(-\tau,T)$, $\tau>0$, is a weak solution on $[0,T)$
by introducing a suitable extension of $u$ in $(-\I,-\tau/2)\times\R^d$.

For a given data $u_0\in L^1_{\mathrm{loc}}(\R^d)$, we define the maximal existence time by
\begin{align*}
 	T_{\max} =
	T_{\max}(u_0)
	 := \sup \left\{T >0\; ;\;
	\begin{aligned}
	&\text{There exists a weak solution $u(t)$}\\
	&\text{to \eqref{eq:NLS} with $u(0) = u_0$ on $[0, T)$}
	\end{aligned}
	\right\}.
\end{align*}
\begin{theorem}[Finite time blowup]\label{thm:FB}
Let $d\ge1$ and $\e >0$. 
Suppose that $\{g_n\}_n \in \ell^1(\Z)$ satisfies $g_0=1$ and
$\mu := g_0- \sum_{n\neq0} |g_n|>0$.
If $f \in L^{1}_{\mathrm{loc}}(\R^d)$ satisfy 
\begin{align}
	-\Im f(x) \ge 
	\left\{
	\begin{aligned}
	&|x|^{-k} && |x| >R_0, \\
	&0 && |x| \le R_0,
	\end{aligned}
	\right.
	\label{assmp:ii}
\end{align}
for some $k\le d$ and $R_0>0$, then there exist $C=C(k,R_0,\mu)>0$ and $\e_0 >0$ such that
\begin{align}
	T_{\max}(\eps f) \le 
	\left\{
	\begin{aligned}
	&C \e^{-\frac{2}{d-k}} && k<d, \\
	&\exp (C/\eps) && k=d
	\end{aligned}
	\right.
	\label{thm:ii1}
\end{align}
holds for any $\e \in (0,\e_0)$. 
\end{theorem}
\begin{remark}
Let us emphasize that uniqueness of a weak solution is not assumed in Theorem \ref{thm:FB}.
The estimate \eqref{thm:ii1} implies that any existence interval of a weak solution obeys the estimate.
\end{remark}
\begin{remark}
In addition to the assumption of the theorem, we suppose $F(e^{i\theta})$ is Lipschitz continuous and $f\in L^2(\R^d)$.
Then, a standard contraction argument yields a unique solution $u(t)$ in the sense of Definition \ref{def:sol}.
Let $I_{\max}$ be a maximal existence interval of the solution.
Then, $T_{\max}:=\sup I_{\max}$ coincides with the above one and
$u(t)$ blows up at $t=T_{\max}$ in such a sense that
$\lim_{t \to T_{\max}-0} \Lebn{u(t)}{2} = \I$.
\end{remark}

The rest of the paper is as follows.
In Section 2, we give an outline of the proof of Theorem \ref{thm:main} and
extract main technical parts of the proof.
Section 3 is devoted to the main parts.
Then, in Section 4, we turn to the proof of Theorem \ref{thm:Barab}.
Finally, Theorem \ref{thm:FB} is discussed in Section 5.


\section{Outline of the proof of Theorem \ref{thm:main}}
Let $u(t)$ be a solution on $[T,\I)$.
By the equation, we have
\begin{equation}\label{eq:int}
	U(-2t)u(2t) -U(-t)u(t) = -i\sum_{n\in \Z} g_n \int_t^{2t} U(-s) F_n(u(s)) ds
\end{equation}
in $L^2(\R^d)$ for $t>T$, where
\[
	F_n(u) = |u|^{1+\frac2d-n}u^n.
\]
Note that the right hand side makes sense as a $L^2(\R^d)$ function
by means of
(dual) Strichartz's estimate and the assumption $\{g_n\}_n \in \ell^1(\Z)$.

Our proof is in the same spirit as in the paper by Shimomura and Tsutsumi \cite{STs}.
We briefly recall the argument.
Their case corresponds to $g_n=\delta_{n0}$. Hence, \eqref{eq:int} is reduced to 
\begin{equation*}
	U(-2t)u(2t) -U(-t)u(t) = -i \int_t^{2t} U(-s) |u|^{1+\frac2d}(s) ds.
\end{equation*}
Suppose \eqref{asmp:3} and \eqref{asmp:4} hold with some $u_+ \not \equiv 0$. 
Then, the left hand side 
converges to zero strongly in $L^2(\R^d)$ as $t\to\I$ by means of the assumption \eqref{asmp:3},
while the assumption \eqref{asmp:2} with $\l=0$, which is equivalent to \eqref{asmp:4}, implies
\[
	\norm{\int_t^{2t} U(-s) |u(s)|^{1+\frac2d} ds}_{L^2} = \norm{\int_1^{2} \frac{1}{(2\sigma)^{1+\frac{d}2}}\left|\wha{u_+}\(\frac{x}{2\sigma}\)\right|^{1+\frac2d} d\sigma}_{L^2} + o(1)
\]
as $t\to\I$. Hence, we obtain a contradiction.

Let us go back to our case.
Since the constant $\l \in \R$ in the assumption \eqref{asmp:2} is not necessarily zero,
the left hand side of \eqref{eq:int} does not necessarily converges to zero strongly in $L^2(\R^d)$ as $t\to\I$.
Furthermore, in the general $\{g_n\}_n$ case, it is not easy to estimate the norm of the right hand side 
of \eqref{eq:int} in $L^2(\R^d)$.

The idea here is to 
look at structure of every term in the both sides of \eqref{eq:int} and derive a contradiction by
considering a pairing with a suitable function.
More precisely, let
\begin{align} \label{eq:HG}
	\begin{aligned} 
	H(t,x) :={}& -i D(t)G(x) \\
	:={}& -i D(t)\left[\int_1^{2} \frac{1}{(2\sigma)^{1+\frac{d}2}}\left|\wha{u_+}\(\frac{\cdot}{2\sigma}\)\right|^{1+\frac2d} d\sigma\right](x).
	\end{aligned}
\end{align}
Then, we have
\begin{align*}
	&\(-i\int_t^{2t} U(-s) F_0(u(s)) ds, H(t)\)_{L^2}\\
	&{}= (U(-2t)u(2t),H(t))_{L^2} - (U(-t)u(t),H(t))_{L^2}\\
	&\quad{}+ i\sum_{n\neq 0} g_n \( \int_t^{2t} U(-s) F_n(u(s)) ds , H(t)\)_{L^2},
\end{align*}
where $(f,g)_{L^2}=\int_{\R^d} f(x) \overline{g(x)}dx$ is the $L^2$ inner product.
The following three lemmas yield a contradiction if $u_+\not\equiv0$.
\begin{lemma}\label{lem:key1}
Suppose that \eqref{asmp:2} holds for some $u_+ \in H^{0,\frac{d}{d+2}}$ and some $\l\in \R$. Then,
\[
	\lim_{t\to\I}\(-i\int_t^{2t} U(-s) F_0(u(s)) ds, H(t)\)_{L^2}\\
	= \norm{G}_{L^2}^2.
\]
\end{lemma}
\begin{lemma}\label{lem:key2}
Suppose that \eqref{asmp:1} holds for some $u_+ \in H^{0,\frac{d}{d+2}}$ and some $\l\in \R$. Then,
$\lim_{t\to0} (U(-\sigma t)u(\sigma t),H(t))_{L^2} = 0$
for $\sigma=1,2$.
\end{lemma}
\begin{lemma}\label{lem:key3}
Suppose that \eqref{asmp:2} holds for some $u_+ \in H^{0,\frac{d}{d+2}}$ and some $\l\in \R$. Then,
\[
	\lim_{t\to0} \sum_{n\neq 0} g_n \( \int_t^{2t} U(-s) F_n(u(s)) ds , H(t)\)_{L^2} = 0.
\]
\end{lemma}
The three lemmas are proved in the forthcoming section.
 
\section{Completion of the proof of Theorem \ref{thm:main}}
\subsection{Summary of property of $V_\l(t)$}
We first collect basic properties on the asymptotic profile $V_\l(t)$ defined in \eqref{eq:Vl}.
\begin{lemma} \label{lem:ap1}
(i) For $p \ge 2$ and $t>0$,
\begin{equation}\label{eq:Vl1}
	\Lebn{V_{\l}(t)}{p} = Ct^{-d\(\frac12 - \frac1p \)}\Lebn{\wha{u_+}}{p}.
\end{equation}
(ii) For $t>0$,
\begin{equation}\label{eq:Vl2}
	t^{\frac{d}{2(d+2)}} \norm{V_{\l}}_{L^{\frac{2(d+2)}{d}}_{t,x}((t,2t)\times \R^d)}
	\lesssim \Lebn{\wha{u_+}}{\frac{2(d+2)}2}
\end{equation}
\end{lemma}
\begin{proof}
It is obvious by definition \eqref{eq:Vl}.
\end{proof}
\begin{lemma}\label{lem:equivalence}
If $u_+ \in H^{0,\frac{d}{d+2}}$ then
\begin{equation}\label{eq:Vl3}
	\lim_{t\to\I} \norm{U(t)u_+ - V_0(t)}_{L^2} =0
\end{equation}
and
\begin{equation}\label{eq:Vl4}
	\lim_{t\to\I} t^{\frac{d}{2(d+2)}}\norm{U(\cdot)u_+ - V_0}_{L^{\frac{2(d+2)}{d}}_{t,x}((t,\I)\times \R^d)}
	=0
\end{equation}
hold.
\end{lemma}
\begin{proof} Let $t>0$.
Note that
\[
	U(t)u_+ - V_0(t)
	= i^{-\frac{d}2} M(t) D(t) \(U\(-\frac1{4t}\) -1\)\wha{u_+}
\]
The first one follows from unitary property of $M(t)$ and $D(t)$, and the continuity property $U(t)u_+ \in C(\R;L^2)$.
By the Sobolev embedding, we have
\begin{align*}
	\norm{U(t)u_+ - V_0(t)}_{L^{\frac{2(d+2)}{d}}_x(\R^d)}
	&{}=(2t)^{-\frac{d}{d+2}}\norm{\(U\(-\frac1{4t}\) -1\)\wha{u_+}}_{L^{\frac{2(d+2)}{d}}_x(\R^d)} \\
	&{}\lesssim t^{-\frac{d}{d+2}}\norm{\(U\(-\frac1{4t}\) -1\)|\n|^{\frac{d}{d+2}} \wha{u_+}}_{L^2}\\
	&{}=o(t^{-\frac{d}{d+2}}),
\end{align*}
from which the second one follows.
\end{proof}

\subsection{Proof of Lemma \ref{lem:key1}}
This part is the same as in \cite{STs}. We give a proof for completeness.
\begin{proof}
Remark that
\[
	H(t) = -i \int_t^{2t} F_0 (V_{\l}(s)) ds.
\]
Hence, we have
\[
	\(-i \int_t^{2t} U(-s)F_0 (u(s)) ds,H(t)\)
	=(I_1(t),H(t))+(I_2(t),H(t))+\norm{H(t)}_{L^2}^2,
\]
where
\[
	I_1(t) = -i \int_t^{2t} U(-s)(F_0 (u(s))-F_0(V_\l(s))) ds
\]
and
\[
	I_2(t) = -i \int_t^{2t} (U(-s)-1)F_0(V_\l(s)) ds.
\]
Noting that $D(t)$ is unitary on $L^2$,
it suffices to show that $I_1(t)$ and $I_2(t)$ converge to zero strongly in $L^2$ as $t\to\I$.

By H\"older's inequality and \eqref{eq:Vl1}, we obtain
\begin{align*}
	\Lebn{I_1(t)}{2}
	\lesssim{}&  \int_{t}^{2t} \Lebn{u(s) -V_{\l}(s)}{\frac{2(d+2)}{d}}^{1+\frac{2}d} ds\\
	&+ \int_t^{2t} \Lebn{V_{\l}(s)}{\frac{2(d+2)}{d}}^{\frac{2}d} \Lebn{u(s) -V_{\l}(s)}{\frac{2(d+2)}{d}} ds
	 ds \\
	\lesssim{}& \( t^{\frac{d}{2(d+2)}} \norm{u - V_{\l}}_{L^{\frac{2(d+2)}{d}}_{t,x}((t,2t)\times \R^d)}\)^{1+\frac{2}{d}}\\
	&+\Lebn{\wha{u_+}}{\frac{2(d+2)}{d}}^{\frac{2}d} t^{\frac{d}{2(d+2)}} \norm{u - V_{\l}}_{L^{\frac{2(d+2)}{d}}_{t,x}((t,2t)\times \R^d)} \to 0
\end{align*}
as $t\to\I$, thanks to the assumption \eqref{asmp:2}.

On the other hand,
since $|\wha{u_+}|^{1+\frac{2}d} \in L^2$ by the assumption $u_+\in H^{0.\frac{d}{d+2}}$,
\begin{align*}
	\Lebn{I_2(t)}{2} \le{}& \int_t^{2t} \Lebn{ \( U(-s) -1 \) |D(s)\wha{u_+}|^{1+\frac{2}d} }{2} ds \\
	={}&  \int_t^{2t} \frac1{2s} \Lebn{ \( U\(-\frac{1}{4s}\) -1 \) |\wha{u_+}|^{1+\frac{2}d}}{2} ds \\
	={}&  \int_1^{2} \frac1{2\s} \Lebn{ \( U\(-\frac{1}{4\s t}\) -1 \) |\wha{u_+}|^{1+\frac{2}d}}{2} d\s 
	\to 0
\end{align*}
by means of continuity of $U(t)$ and Lebesgue's convergence theorem.
\end{proof}

\subsection{Proof of Lemma \ref{lem:key3}}
We next prove Lemma \ref{lem:key3}.
\begin{proof}
First, we see from \eqref{asmp:2} and \eqref{eq:Vl2} that
there exists $T>0$ independent of $n$ such that
\begin{align*}
	\abs{\( \int_t^{2t} U(-s) F_n(u(s)) ds , H(t)\)_{L^2}}
	&{}\lesssim  \(t^{\frac{d}{2(d+2)}}\norm{u}_{L^{\frac{2(d+2)}d}_{t,x}((t,2t)\times \R^d)}\)^{1+\frac2d} \norm{G}_{L^2} \\
	&{}\lesssim (1+ \norm{\wha{u_+}}_{L^{\frac{2(d+2)}d}})^{1+\frac2d} \norm{G}_{L^2}
\end{align*}
for any $n$ and any $t\ge T$. Since $\{g_n\}_n \in\ell^1(\Z)$,
by means of Lebesgue's convergence theorem (in $n$),
it suffices to show that
\[
\( \int_t^{2t} U(-s) F_n(u(s)) ds , H(t)\)_{L^2} \to 0
\]
as $t\to\I$ for each fixed $n\neq0$.

Fix $n\neq 0$. Then, one has
\[
	\( \int_t^{2t} U(-s) F_n(u(s)) ds , H(t)\)_{L^2}
	=\( \int_t^{2t} U(-s)F_n(V_\l(s)) ds , H(t)\)_{L^2} + I_{1,n}(t),
\]
where
\[
	I_{1,n}(t) = -i \int_t^{2t} U(-s)(F_n (u(s))-F_n(V_\l(s))) ds.
\]

Remark that
\[
	|F_n(z_1)-F_n(z_2)| \le C|n|(|z_1-z_2|^{1+\frac2d} + |z_1|^{\frac2d}|z_1-z_2|),\quad \forall z_1,z_2 \in \C.
\]
Hence, just as in the proof of Lemma \ref{lem:key1}, we obtain $I_{1,n}(t) \to 0$
as $t\to\I$.

Let $E(t)=e^{it|x|^2}$. A computation shows
\begin{align*}
&\( \int_t^{2t} U(-s)F_n(V_\l(s)) ds , H(t)\)_{L^2}\\
&=c_n\int_1^{2} \( E(n\sigma t) e^{-in\l |\wha{u_+}|^{\frac2d}\log \sigma t}F_n(\wha{u_+}), D\(\frac{1}{2\sigma}\)U\(\frac{\sigma}{4t}\)G \)_{L^2} \frac{ds}{2\sigma}
\end{align*}
where $c_n \in \C$ is a constant such that $|c_n|=1$.
As the integrand is bounded by $\norm{\wha{u_+}}_{L^{\frac{2(d+2)}{d}}}^{1+\frac2d} \norm{G}_{L^2} \in L^1_\sigma((1,2))$,
we shall show it converges to zero as $t\to\I$ for each $\sigma \in (1,2)$.
Since $U(t)\to \mathrm{Id}$ strongly as $t\to0$, we shall show
\begin{equation}\label{eq:pf1}
	\( E(n\sigma t) e^{-in\l |\wha{u_+}|^{\frac2d}\log \sigma t}F_n(\wha{u_+}), D\(\frac{1}{2\sigma}\)G \)_{L^2} \to 0
\end{equation}
as $t\to\I$. 

We prove \eqref{eq:pf1}.
Fix $\sigma \in [1,2]$.
Set $\f(t) = -\l n|\wha{u_+}|^{\frac{2}{d}} \log \sigma t$.
By density argument, we may approximate $D(1/2\sigma)G \in L^2(\R^d)$ by $\v \eta$, where $\eta \in C_0^{\I}(\R^d, \C)$ and $\v \in C_0^{\I}(\R^d, \R)$ is a nonnegative radial cutoff such that $\supp \v \subset \{x \in \R^d;\; \delta^{-1} \le |x| \le \delta\}$ for $\delta \gg 1$. Thus, it suffices to show
\begin{equation}\label{eq:pf1_1}
\( E(n\sigma t) e^{i\f(t)}F_n(\wha{u_+}), \v \eta \)_{L^2} \to 0
\end{equation}
as $t\to\I$.

Let $\psi(r) \in C_0^\I(\R)$ be another nonnegative radial cutoff such that $\psi(r)=1$ on $0\le r \le 1$
and $\psi(r)=0$ for $r\ge2$. 
Then, $\chi(t):= \psi (t^{-1/2}|\n|):=\F^{-1} \psi(t^{-1/2}|\xi|)\F$ is a time dependent regularizing operator.
The left hand side of \eqref{eq:pf1_1} is written as
\begin{align*}
& \( E(n\sigma t)  \chi(t) (e^{i\f(t)}F_n(\wha{u_+}) \overline{\eta}),\v \)_{L^2}
+\( E(n\sigma t) (1-\chi(t)) (e^{i\f(t)}F_n(\wha{u_+})\overline{\eta}) , \v\)_{L^2}\\
&{}=: I_{3,n}(t) + I_{4,n}(t).
\end{align*}

Let us first estimate $I_{3,n}(t)$.
By integration by parts,
\[
	I_{3,n} (t)= -\int e^{in\sigma t |x|^2} \nabla \cdot \( \frac{x}{2in\sigma t |x|^2} (\chi(t) e^{i\f(t)}F_n(\wha{u_+})\overline{\eta})(x)\v(x)  \) dx
\]
Using $\supp \v, \supp |\n\v| \subset \{ |x| \ge \delta^{-1}\}$ and $\norm{\n \chi(t)}_{\mathcal{L}(L^2)}\lesssim t^{1/2}$,
one sees that $|I_{3,n}(t)|\to0$ as $t\to\I$.

We move to the estimate of $I_4(t)$. Let $\tilde{d}:=\max(3,d)$.
A use of H\"older's inequality gives us
\begin{align*}
	|I_{4,n}(t)| \le C\norm{(1-\chi(t))e^{i\f(t)}F_n(\wha{u_+})\overline{\eta}}_{L^{\frac{2\tilde{d}}{\tilde{d}+2}}} \norm{\v}_{L^{\frac{2\tilde{d}}{\tilde{d}-2}}}.
\end{align*}
One sees from Mihlin's multiplier theorem (see \cite[Theorem 5.2.7]{GR}) that
$\norm{|\nabla|^{-\theta}(1-\chi(t))}_{\mathcal{L}(L^p)}\lesssim t^{-\theta/2}$ for any $p\in(1,\I)$ and $\theta>0$.
Combining this with a fractional Leibniz rule, we obtain
\begin{align*}
	&{}\norm{(1-\chi(t))e^{i\f(t)}F_n(\wha{u_+})\overline{\eta}}_{L^{\frac{2\tilde{d}}{\tilde{d}+2}}} \\
	&{}\lesssim t^{-\frac{\theta_0}{2}} \( \norm{|\n|^{\theta_0} \( (e^{i\f(t)}-1) F_n(\wha{u_+})\overline{\eta}\)}_{L^{\frac{2\tilde{d}}{\tilde{d}+2}}} + \norm{|\n|^{\theta_0} \( F_n(\wha{u_+})\overline{\eta}\)}_{L^{\frac{2\tilde{d}}{\tilde{d}+2}}} \) \\
	&{}\lesssim t^{-\frac{\theta_0}{2}} \Lebn{|\n|^{\theta_0} (e^{i\f(t)}-1)}{p_1} \Lebn{F_n(\wha{u_+})}{2} \Lebn{\eta}{p_2} \\
	&\quad + t^{-\frac{\theta_0}{2}} \Lebn{e^{i\f(t)} -1}{p_2} \Lebn{|\n|^{\theta_0} F_n (\wha{u_+})}{p_3} \Lebn{\eta}{p_4} 
	\\
	&\quad + t^{-\frac{\theta_0}{2}}\Lebn{e^{i\f(t)} -1}{p_2}
	\Lebn{F_n(\wha{u_+})}{2} \Lebn{|\n|^{\theta_0}\eta}{p_5} \\
	&\quad + t^{-\frac{\theta_0}{2}} \Lebn{|\n|^{\theta_0} F_n (\wha{u_+})}{p_3} \Lebn{\eta}{p_6} + t^{-\frac{\theta_0}{2}} 	\Lebn{F_n(\wha{u_+})}{2} \Lebn{|\n|^{\theta_0}\eta}{\tilde{d}},
\end{align*}
where the exponents $\theta_0$, $p_1$, $p_2$, $p_3$, and $p_4$ are defined as follows:
\begin{align*}
	\theta_0 =& \frac{d}{10\tilde{d}(d+2)},&
	\frac1{p_1}=&\frac{10d+1}{10\tilde{d}(d+2)},&
	\frac1{p_2}=&\frac{2}{\tilde{d}(d+2)}, \\
	\frac1{p_3}=&\frac{5\tilde{d}(d+2)+1}{10\tilde{d}(d+2)},&
	\frac1{p_4}=&\frac{10d-1}{10\tilde{d}(d+2)}, &
	\frac1{p_5}=&\frac{d}{\tilde{d}(d+2)}, \\
	\frac1{p_6}=&\frac{19}{10\tilde{d}(d+2)}.
\end{align*}

We now recall the following property.

\begin{proposition}[{\cite[Proposition A.1]{MR2318286}}]\label{prop:Visan}
Let $F$ be a H\"older continuous function of order $\a \in (0,1)$. Then for every $0 < \s < \a$, $1<p<\I$, and $\frac{\s}{\a} < s<1$, we have
\[
	\Lebn{|\n|^{\s} F(u)}{p} \le C \Lebn{|u|^{\a - \frac{\s}{s}}}{p_1}\Lebn{|\n|^{s}u}{\frac{\s}{s}p_2}^{\frac{\s}{s}},
\]
provided $\frac{1}{p} = \frac{1}{p_1} + \frac{1}{p_2}$ and $\( 1-\frac{\s}{\a s}\)p_1 >1$.
\end{proposition}

Let $\alpha_0=\alpha_0(d):=2/\tilde{d}$.
Since $e^{i\phi(t)}-1$ is a $\alpha_0$-H\"older function (of $\wha{u_+}$),
we see from Proposition \ref{prop:Visan} that
\begin{align*}
	\norm{|\n|^{\theta_0}(e^{i\f(t)}-1)}_{L^{p_0}}
	&{}\lesssim (|\l n|\log(t \s))^{\frac{d}2 \alpha_0}
	\norm{\wha{u_+}}_{L^{\frac{2(d+2)}{d}}}^{\alpha_0-\frac1{10\tilde{d}}}
	\norm{|\n|^{\frac{d}{d+2}}\wha{u_+}}^{\frac1{10\tilde{d}}}_{L^{2}}\\
	&{}\lesssim_n (\log(t \s))^{\frac{d}2 \alpha_0}\norm{|\n|^{\frac{d}{d+2}}\wha{u_+}}^{\alpha_0}_{L^{2}}.
\end{align*}
Further, using the Sobolev embedding, we deduce that
\begin{align*}
	\Lebn{e^{i\f(t)} -1}{p_2} &{} \lesssim (|\l n|\log(t \s))^{\frac{d}2 \alpha_0} \Lebn{\wha{u_+}}{d+2}^{\a_0} \\
	&{} \lesssim_{n} (\log(t \s))^{\frac{d}2 \alpha_0} \Lebn{|\n|^{\frac{d}{d+2}} \wha{u_+}}{2}^{\a_0}.   
\end{align*}

On the other hand, arguing as in \cite[Lemma 2.4]{MMU} and \cite[Lemma 3.7]{MS}, we have
\begin{equation}\label{eq:pf1_2}
\begin{aligned}
	\Lebn{|\n|^{\theta_0} F_n (\wha{u_+})}{p_3}
	&{}\lesssim |n|^{1+\frac2d} \Lebn{\wha{u_+}}{\frac{2(d+2)}{d}}^{\frac2d}\Lebn{|\n|^{\theta_0}\wha{u_+}}{\frac{10\tilde{d}(d+2)}{5d\tilde{d}+1}}\\
	&{}\lesssim_n \norm{|\n|^{\frac{d}{d+2}}\wha{u_+}}^{1+\frac2d}_{L^{2}}.
\end{aligned}
\end{equation}
Combining these estimates, we conclude that $I_{4,n}(t)\to0$ as $t\to\I$.
\end{proof}

\subsection{Proof of Lemma \ref{lem:key2}}
\begin{proof}
In view of \eqref{asmp:1}, it suffices to show that
\[
	\lim_{t \to \I} \(U(-\sigma t)V_{\l}(\sigma t), D(t)G \)_{L^2} =0.
\]
To this end, we first note that
\begin{align*}
	&\(U(-\sigma t)V_{\l}(\sigma t), D(t)G \)_{L^2} \\
	&{}=  i^{-\frac{d}{2}} \( E(\sigma t) e^{-i\l |\wha{u_+}|^{\frac2d}\log \sigma t}\wha{u_+}, D\(\frac{1}{2\sigma}\)\( U\(\frac{\sigma }{4t}\) -1\) G \)_{L^2} \\
	&{}\quad + i^{-\frac{d}{2}} \( E(\sigma t) e^{-i\l |\wha{u_+}|^{\frac2d}\log \sigma t}\wha{u_+}, D\(\frac{1}{2\sigma}\) G \)_{L^2}.
\end{align*}
The first term of the right hand side tends to zero as $t\to\I$ because of strong continuity of $U(t)G$.
By essentially the same argument as in \eqref{eq:pf1} for $n=1$, we see that the second term also tends to zero as $t\to\I$.
The only difference is that $F_1$ is replaced by $\wha{u}_+$ and that $p_3$, $p_4$ and $p_6$ are replaced by
\[
	\frac1{\widetilde{p_3}}= \frac{5d\tilde{d}+1}{10\tilde{d}(d+2)}, \quad
	\frac1{\widetilde{p_4}}= \frac{10d+10\tilde{d}-1}{10\tilde{d}(d+2)}, \quad
	\frac1{\widetilde{p_6}}= \frac{10d+10\tilde{d}+19}{10\tilde{d}(d+2)}, \quad
\]
respectively. By the choice, the estimate \eqref{eq:pf1_2} is replaced by
\[
	\norm{ |\nabla|^{\theta_0} \wha{u_+}}_{L^{\widetilde{p_3}}}
	\lesssim \norm{|\n|^{\frac{d}{d+2}} \wha{u_+}}_{L^2},
\]
which is acceptable.
\end{proof}

\subsection{Generalization of Theorem \ref{thm:main}}

It would be clear from the above proof that our argument can be applied to more types of behavior.
Here, we take a real-valued function $\phi(t,x)$ and
consider the asymptotic profile $V_\phi (t,x)$ of the following form
\begin{equation}\label{eq:Vphi}
	V_\phi (t,x) = e^{-i\frac{\pi}4 d} M(t) D(t) [e^{i\phi(t)}\wha{u_+}](x).
\end{equation}
One sees that our proof works if the property corresponding to \eqref{eq:pf1} is true.
Hence, we introduce the following assumption on the phase function $\phi$.
\begin{assumption}\label{asmp:gphase}
$\phi(t,x)$ is a real-valued function.
Suppose there exist positive numbers $a$ and $b$, $a<b$, such that
for any $f,g \in L^2(\R^d)$, $\sigma \in [a,b]$, and $n\neq 0$, it holds that
\[
	\lim_{t\to\I}\( e^{in(\sigma t|x|^2 + \phi(\sigma t))} f, g \)_{L^2} = 0.
\]
\end{assumption}
Intuitively, this assumption implies that $\phi$ does not cancel out oscillation.
A simple counter example is $\phi (t,x) =-t|x|^2 {\bf 1}_{\{ |x| \le 1\}}$.
For various types of phase, the assumption can be justified by the stationary phase.
\begin{theorem}\label{thm:general}
Let $d\ge1$.
Suppose that $\{g_n\}_n \in \ell^1(\Z)$ and $g_0 =1$. 
Suppose that $\phi$ satisfies Assumption \ref{asmp:gphase}.
If a solution $u(t)$ to \eqref{eq:NLS} on $[T,\I)$, $T\in \R$, satisfies
\begin{align*}
\lim_{t \to \I} \Lebn{u(t)-V_\phi(t)}{2} = 0, \\
\lim_{t \to \I} t^{\frac{d}{2(d+2)}} \norm{u(\cdot) - V_\phi(\cdot)}_{L^{\frac{2(d+2)}{d}}_{t,x}([t, \I)\times \R^d)} = 0,
\end{align*}
for some $u_+ \in H^{0,\frac{d}{d+2}}(\R^d)$, where $V_\phi(t)$ is given in \eqref{eq:Vphi},
then $u_{+} \equiv 0$.
\end{theorem}

\section{Proof of Theorem \ref{thm:Barab}}
The strategy of the proof is similar to in Theorem \ref{thm:main}.
This argument can be compared with that in \cite{B,St}.
\begin{proof}
We consider a pairing of \eqref{eq:int} and $u_+$:
\begin{align}
	\begin{aligned}
	&\(-ig_1 \int_t^{2t} U(-s) F_1(u(s)) ds, u_+ \)_{L^2}\\
	&{}= (U(-2t)u(2t) -U(-t)u(t), u_+)_{L^2} \\
	&\quad{}+ i\sum_{n\neq 0, 1} g_n \( \int_t^{2t} U(-s) F_n(u(s)) ds , u_+ \)_{L^2}.
	\end{aligned}
	\label{bara:1}
\end{align}
By assumption \eqref{asmp:3}, the first term of the right hand side tends to zero as $t\to\I$.
Hence, we shall show
\begin{align}
	\lim_{t \to \I}\(\int_t^{2t} U(-s) F_1(u(s)) ds, u_+ \)_{L^2} =  \frac{\log 2}{2} \Lebn{\wha{u_+}}{\frac{2(d+1)}{d}}^{\frac{2(d+1)}{d}}. \label{bara:2}
\end{align}
and 
\begin{align}
	\sum_{n\neq 1} g_n \( \int_t^{2t} U(-s) F_n(u(s)) ds , u_+ \)_{L^2} \to 0 \label{bara:3}
\end{align}
as $t\to\I$. These estimates show $u_+\equiv0$.

Let us begin with \eqref{bara:2}.
Note that
\[
	\(\int_t^{2t} U(-s) F_1(u(s)) ds, u_+ \)_{L^2} =  \int_t^{2t} \( F_1(u(s)), U(s)u_+ \)_{L^2} ds.
\]
By an estimate similar to that of $I_1(t)$ in the proof of Lemma \ref{lem:key1}
and by \eqref{eq:Vl3}, we have
\[
	\int_t^{2t} \( F_1(u(s)), U(s)u_+ \)_{L^2} ds =
	\int_t^{2t} \( F_1(V_0(s)), V_0(s) \)_{L^2} ds  + o(1)
\]
as $t\to\I$.
A computation yields
\begin{align*}
	\int_t^{2t} \( F_1(V_0(s)), V_0(s) \)_{L^2} ds
	&{}= \int_t^{2t} \( F_1(\wha{u_+}), \wha{u_+} \)_{L^2} \frac{ds}{2s}
	= \frac{\log 2}{2} \norm{\wha{u_+}}_{L^{\frac{2(d+1)}d}}^{\frac{2(d+1)}d},
\end{align*}
which completes the proof of \eqref{bara:2}.

Let us next prove \eqref{bara:3}.
We see from \eqref{asmp:2} and \eqref{eq:Vl2} that
there exists $T>0$ independent of $n$ such that
\begin{align*}
	\abs{\( \int_t^{2t} U(-s) F_n(u(s)) ds , u_+\)_{L^2}}
	&{}\lesssim  \(t^{\frac{d}{2(d+2)}}\norm{u}_{L^{\frac{2(d+2)}d}_{t,x}((t,2t)\times \R^d)}\)^{1+\frac{2}{d}}\norm{u_+}_{L^2} \\
	&{}\lesssim (1+ \norm{\wha{u_+}}_{L^{\frac{2(d+2)}d}}^{1+\frac2{d}}) \norm{u_+}_{L^2}
\end{align*}
for any $n$ and any $t\ge T$. 
Hence, by means of $\{g_n\}_n \in \ell_1$, it suffices to show
\begin{align*}
	\( \int_t^{2t} U(-s) F_n(u(s)) ds , u_+ \)_{L^2} \to 0
\end{align*}
as $t \to \I$ for each $n \neq 0, 1$.
Arguing as in the proof of \eqref{bara:2}, we obtain
\[
	\( \int_t^{2t} U(-s) F_n(u(s)) ds , u_+ \)_{L^2}
	=\int_t^{2t} \( F_n(V_0(s)), V_0(s) \)_{L^2} ds + o(1)
\]
as $t\to\I$. Remark that
\begin{align*}
	&\int_t^{2t} \( F_n(V_0(s)), V_0(s) \)_{L^2} ds \\
	&{}= e^{-i\frac{(n-1)d}4}\int_1^{2} \( E((n-1)t\sigma)|\wha{u_+}|^{1+\frac2d-n}\wha{u_+}^n, \wha{u_+} \)_{L^2} \frac{d\s}{2\sigma}.
\end{align*}
As in the proof of \eqref{eq:pf1}, integration by parts with a standard density argument shows
this term tends to zero as $t\to\I$ as long as $n\neq1$.
\end{proof}

\section{Proof of Theorem \ref{thm:FB}}
We follow the test function method argument as in \cite{II2}.
\begin{proof}
Let $f\in L^1_{\mathrm{loc}}(\R^d)$ and
let $u(t,x)$ be a weak solution on $[0,T_0)$ with initial condition $u(0)=\eps f$. 
We may suppose that $T_0 > \frac12 T_{\max}(\eps f)$.

Set $\f(x) = \exp(1-\sqrt{1+|a x|^2})$, where $a>0$ is the number
such that $\int_{\R^d}\f =1$.
Remark that $\f(0) =1$ and there exists $M >0$ such that $|\D \f(x)| \le M \f(x)$ for all $x \in \R^d$.
We next set 
\begin{align*}
	\eta(t) = 
	\left\{
	\begin{aligned}
	&0 && t >1, \\
	&(1-t)^{\t} && 0 \le t \le 1,
	\end{aligned}
	\right.
\end{align*}
where $\theta\ge 1+d/2$. Remark that there exists $N >0$ such that
$|\d_t \eta| \le N |\eta|^{d/(d+2)}$ holds for all $t> 0$.
Then,  for any $R>0$ we denote
\begin{align*}
	\psi_R(t,x) := \eta_R(t)\f_R(x), \quad
	\eta_R (t) := \eta\(\frac{t}{R^2}\), \quad \f_R\(x\) := \f\(\frac{x}{R}\).
\end{align*} 
By density argument, we have
\begin{equation}\label{eq:pf3_1}
\begin{aligned}
	&{}\int_{[0, R^2) \times \R^d} u(t,x) \{ -i\pa_t (\psi_R(t,x)) + \D(\psi_R(t,x)) dt dx \\
	= {}& i\e \int_{\R^d} f(x) \f_R(x) dx + \sum_{n \in \Z}g_n\int_{[0,R^2) \times \R^d} F_n(u(t,x)) \psi_R(t,x) dtdx   
\end{aligned}
\end{equation}
for any $R>0$ such that $R^2 < T_0$. 
Then, the following is the key.
\begin{lemma} \label{lem:ii1}
There exists a constant $C=C(d,\mu, M,N)>0$ such that
\begin{align}
	-\e \int_{\R^d}\Im f(x) \f_R(x) dx \le C 
	, \label{ii:est1}
\end{align}
holds as long as $R^2 < T_0$.
\end{lemma}
Once we obtain \eqref{ii:est1}, the proof is straight forward.
Suppose $T_{\max}(\eps f)>1$ for some $\eps>0$.
If such $\eps>0$ does not exist, the estimate \eqref{thm:ii1} is trivial with $C=\eps_0=1$.
The assumption on $f$ gives us
\[
	\eps R^{d-k}\int_{|x|\ge R_0/R} |x|^{-k}\phi(x) dx \le -\e \int_{\R^d}\Im f(x) \f_R(x) dx
\]
for any $R>1/2$. Further, for any $R>1/2$, one has
\[
	\int_{|x|\ge R_0/R} |x|^{-k}\phi(x) dx \ge 
	\int_{|x|\ge 2R_0} |x|^{-k}\phi(x) dx =:C(k,R_0)
\]
if $k<d$ and
\[
	\int_{|x|\ge R_0/R} |x|^{-k}\phi(x) dx \ge C(k,R_0) \log R
\]
if $k=d$.
Plugging these estimates to
\eqref{ii:est1} with $R^2=\frac12 T_0>\frac14 T_{\max}(\eps f)$, 
we obtain \eqref{thm:ii1} with a constant $C=C(k,R_0,\mu_0)$.
Now, we chose $\eps_0>0$ so that 
the right hand side of \eqref{thm:ii1} is equal to two
with this constant $C$. Then, \eqref{thm:ii1} is true for all $\eps \in (0,\eps_0)$.
\end{proof}
\begin{proof}[Proof of Lemma \ref{lem:ii1}]
Let us introduce
\begin{align*}
	I_{n}(R) = \int_{[0,R^2) \times \R^d} F_n(u(t,x)) \psi_R(t,x) dx dt, \quad J(R) = \int_{\R^d} f(x) \f_R(x) dx. 
\end{align*}
Comparing real part of the both sides of \eqref{eq:pf3_1} and making a use of specific choice of $\psi_R$,
one deduces from H\"older's inequality that
\[
	-\e \Im J(R) + I_{0}(R) + \Re\( \sum_{n \neq 0}g_n I_{n}(R)\) 
	\le C_{d,M,N} I_{0}(R)^{\frac{d}{d+2}}.
\]
In view of $|I_{n}(R)| \le I_{0}(R)$, we have
\begin{align*}
	-\e \Im J(R) \le C_{d,M,N} I_{0}(R)^{\frac{d}{d+2}} - \mu I_{0}(R) \le C(d,\mu,M,N)
\end{align*}
as claimed.
\end{proof}
\subsection*{Acknowledgments} 
This work was done while the authors
were visiting at Department of Mathematics at the University of California,
Santa Barbara whose hospitality they gratefully acknowledge.
S.M. was supported by the Sumitomo Foundation Basic Science Research
Projects No.\ 161145 and by JSPS KAKENHI Grant Numbers JP17K14219, JP17H02854, and JP17H02851.
H.M. was supported by the Overseas Research Fellowship Program by National Institute of Technology.

\bibliographystyle{amsplain}

\end{document}